\documentclass{amsart}
\usepackage{amssymb,latexsym}
\usepackage{amscd,amsthm}
\usepackage{tikz-cd}

\usepackage[all]{xy}

\newtheorem{theorem}{Theorem}[section]
\newtheorem{lemma}[theorem]{Lemma}
\newtheorem{proposition}[theorem]{Proposition}
\newtheorem{corollary}[theorem]{Corollary}

\theoremstyle{definition}
\newtheorem{definition}[theorem]{Definition}

\newtheorem*{remark}{Remark}

\DeclareMathOperator{\Ext}{Ext}

\DeclareMathOperator{\colim}{colim}

\DeclareMathOperator{\cok}{cok}

\newcommand{\cat}[1]{\mathcal{#1}}           

\newcommand{\class}[1]{\mathcal{#1}}   

\newcommand{\mathcolon}{\colon\,} 

\newcommand{\qcox}{\textnormal{Qco}(X)}
\newcommand{\chqcox}{\textnormal{Ch}(\textnormal{Qco}(X))}

\newcommand{\ch}{\textnormal{Ch}(R)}
\newcommand{\cha}[1]{\textnormal{Ch}(\mathcal{#1})}

\newcommand{\tilclass}[1]{\widetilde{\class{#1}}}
\newcommand{\dwclass}[1]{dw\widetilde{\class{#1}}}
\newcommand{\exclass}[1]{ex\widetilde{\class{#1}}}
\newcommand{\dgclass}[1]{dg\widetilde{\class{#1}}}

\newcommand{\rightperp}[1]{#1^{\perp}}
\newcommand{\leftperp}[1]{{}^\perp #1}

\begin{document}

\title{Models for mock homotopy categories of projectives}

\author{James Gillespie}
\address{Ramapo College of New Jersey \\
         School of Theoretical and Applied Science \\
         505 Ramapo Valley Road \\
         Mahwah, NJ 07430}
\email[Jim Gillespie]{jgillesp@ramapo.edu}
\urladdr{http://pages.ramapo.edu/~jgillesp/}

\keywords{model structure, homotopy category, chain complexes, recollement}
\date{\today}
\subjclass[2010]{55U35, 18G25, 18E35}

\begin{abstract}
Let $R$ be a ring and $\ch$ the category of chain complexes of $R$-modules. We put an abelian model structure on $\ch$ whose homotopy category is equivalent to $K(Proj)$, the homotopy category of all complexes of projectives. However, the cofibrant objects are not complexes of projectives, but rather all complexes of flat modules. The trivial objects are what Positselski calls contraacyclic complexes and so the homotopy category coincides with his contraderived category. We in fact construct this model on the category of chain complexes of quasi-coherent sheaves on any scheme $X$ admitting a flat generator. In this case the homotopy category recovers what Murfet calls the mock homotopy category of projectives. In the same way we construct a model for the (mock) projective stable derived category, and we use model category methods to recover the recollement of Murfet. Finally, we consider generalizations by replacing the flat cotorsion pair with other complete hereditary cotorsion pairs in Grothendieck categories.   
\end{abstract}

\maketitle

\section{Introduction}\label{sec-introduction}

Let $X$ be a separated noetherian scheme and  Qco($X$) the category of quasi-coherent sheaves on $X$. In Murfet's thesis~\cite{murfet}, we saw the introduction of the \emph{mock homotopy category of projectives} $K_m(Proj X)$. While the category Qco($X$) typically doesn't have enough projectives, Murfet showed that this triangulated category $K_m(Proj X)$, which is a certain Verdier localization of the homotopy category of flat sheaves, fills the role one would expect from a homotopy category of projectives. In particular, as Neeman originally showed in~\cite{neeman-flat}, $K_m(Proj X)$ agrees with the usual homotopy category of projectives in the case that $X$ is affine. And, Murfet's thesis shows in detail how $K_m(Proj X)$ plays a fundamental role in an extension of classical Grothendieck duality.

It is an interesting question whether or not we can describe these phenomena using cotorsion pairs and abelian model categories. Several authors have recently considered this question, for example, see~\cite{enochs-model strucs, stovicek-exact model cats, gillespie-recollements2}.  In particular, {\v{S}}\v{t}ov{\'{\i}}{\v{c}}ek~\cite{stovicek-exact model cats} and Gillespie~~\cite{gillespie-recollements2} each constructed model structures on the category of complexes of flat sheaves whose homotopy category is equivalent to $K_m(Proj X)$. But the category of complexes of flat sheaves does not have all small limits and colimits and so the resulting \emph{exact model structures} do not give us a model category in the usual sense. A closer inspection shows that we would like a model structure on the entire chain complex category such that:  (1) The cofibrant objects are exactly the complexes of flat quasi-coherent sheaves, and, (2) In the affine case, its homotopy category recovers Positselski's contraderived category (from~\cite{positselski}), which in turn is equivalent to $K(Proj)$, the homotopy category of all complexes of projectives. Analogous to the usual derived category which is obtained as a Verdier quotient by killing the exact (or acyclic) complexes, the contraderived category is the Verdier quotient obtained by killiing the contraacyclic complexes. By definition, these are the complexes $W$ such that every chain map $P \xrightarrow{} W$ is null homotopic whenever $P$ is a complex of projectives.



So the first goal of this paper is to construct such a model structure, satisfying the conditions (1) and (2) above. Its homotopy category corresponds exactly to Murfet's mock homotopy category of projectives, $K_m(Proj X)$. In fact, if we restrict the model to the Quillen equivalent model inherited by the fully exact subcategory of cofibrant objects, that is, the exact category of chain complexes of flat sheaves, then the resulting exact model structure recovers the previous models constructed in~\cite{enochs-model strucs, stovicek-exact model cats, gillespie-recollements2}. 
The surprising thing is that the model structure we seek is now quite easy to construct due to a new method for constructing abelian model structures that appeared in~\cite{gillespie-hovey triples}. This method works well in situations where we do not know the exact nature of the trivial objects. For example, when looking to construct a model structure for $K_{m}(Proj X)$, we know that a complex $F$ of flat sheaves ought to be trivial whenever it is exact and has each cycle $Z_nF$ flat. But for complexes without flat components, it is not clear (for non affine schemes) what the trivial complexes should be. How can one construct a model structure when one doesn't have an explicit description of the trivial objects? The paper~\cite{gillespie-hovey triples} provides the answer we need to this question. We apply it in Section~\ref{sec-models for mock} to easily obtain, using already known completeness results for cotorsion pairs of complexes, the desired model structure for $K_{m}(Proj X)$. Using the same method we will obtain a model structure for Murfet's \emph{mock projective stable derived category}, denoted $K_{m,ac}(Proj X)$. 

But the much more ambitious goal of this paper is give a model category description of the recollement from Murfet's thesis involving $K_m(Proj X)$, $K_{m,ac}(Proj X)$, and the usual derived category $\class{D}(X)$. This recollement is dual to  the injective recollement of Krause which was interpreted via model category methods by Becker in~\cite{becker}. This method was generalized in~\cite{gillespie-recollements}, and in the current paper we see in Theorems~\ref{them-right recollement} and~\ref{them-left recollement} similar, but more general results, allowing one to obtain recollement situations from abelian model structures. This method is purely (model) categorical and by stripping away all of the algebraic geometry we see that we can obtain the recollement whenever $X$ is a scheme for which Qco($X$) admits a flat generator. For example, this is true whenever $X$ is a quasi-compact and semi-separated scheme. 
In fact the method is so general that we show in the last section how to obtain a similar recollement by starting with any hereditary cotorsion pair $(\class{A},\class{B})$ in a Grothendieck category $\cat{G}$, and as long as it is cogenerated by a set $\class{S}$, and that $\class{A}$ contains a generating set $\{U_i\}$ for $\cat{G}$.


\section{Notation and preliminaries}\label{sec-notation and preliminaries}

For the convenience of the reader we summarize some essential definitions in this section. This includes cotorsion pairs and abelian model structures, recollement situations, and deconstructibility.
Throughout this section and Section~\ref{sec-recollement} we let $\cat{A}$ denote a fixed abelian category. 

\subsection{Cotorsion pairs and abelian model structures}\label{subsec-abelian model cats}

Hovey showed in~\cite{hovey} that an abelian model category $\cat{A}$ is nothing more than two nicely related cotorsion pairs in $\cat{A}$.  By definition, a pair of classes $(\class{X},\class{Y})$ in $\cat{A}$ is called a \emph{cotorsion pair} if $\class{Y} = \rightperp{\class{X}}$ and $\class{X} = \leftperp{\class{Y}}$. Here, given a class of objects $\class{C}$ in $\cat{A}$, the right orthogonal  $\rightperp{\class{C}}$ is defined to be the class of all objects $X$ such that $\Ext^1_{\cat{A}}(C,X) = 0$ for all $C \in \class{C}$. Similarly, we define the left orthogonal $\leftperp{\class{C}}$. We call the cotorsion pair \emph{hereditary} if $\Ext^i_{\cat{A}}(X,Y) = 0$ for all $X \in \class{X}$, $Y \in \class{Y}$, and $i \geq 1$. The cotorsion pair is \emph{complete} if it has enough injectives and enough projectives. This means that for each $A \in \cat{A}$ there exist short exact sequences $0 \xrightarrow{} A \xrightarrow{} Y \xrightarrow{} X \xrightarrow{} 0$ and $0 \xrightarrow{} Y' \xrightarrow{} X' \xrightarrow{} A \xrightarrow{} 0$ with $X,X' \in \class{X}$ and $Y,Y' \in \class{Y}$. 
Besides their connection to abelian model structures which we describe next, cotorsion pairs are fundamental in modern homological algebra. There are several good references. In particular we will refer to~\cite{enochs-jenda-book} and~\cite{hovey}.

The main theorem of~\cite{hovey} showed that an abelian model structure on $\cat{A}$ is equivalent to a triple $(\class{C},\class{W},\class{F})$ of classes of objects in $\cat{A}$ for which $\class{W}$ is thick and $(\class{C} \cap \class{W},\class{F})$ and $(\class{C},\class{W} \cap \class{F})$ are each complete cotorsion pairs. By \emph{thick} we mean that the class $\class{W}$ is closed under retracts and satisfies that whenever two out of three terms in a short exact sequence are in $\class{W}$ then so is the third. In this case, $\class{C}$ is precisely the class of cofibrant objects of the model structure, $\class{F}$ are precisely the fibrant objects, and $\class{W}$ is the class of trivial objects. We hence denote an abelian model structure $\class{M}$ as a triple $\class{M} = (\class{C},\class{W},\class{F})$ and for short 
we will denote the two associated cotorsion pairs above by $(\tilclass{C},\class{F})$ and $(\class{C},\tilclass{F})$. We say that $\class{M}$ is \emph{hereditary} if both of these associated cotorsion pairs are hereditary. We will also call any abelian model structure $\class{M} = (\class{C},\class{W},\class{F})$ a \emph{Hovey triple}.

Note that given any Hovey triple $\class{M} = (\class{C},\class{W},\class{F})$ there are four approximation sequences associated to $\class{M}$. In particular, using enough injectives of $(\tilclass{C},\class{F})$ corresponds to the fibrant replacement functor denoted by $R$. On the other hand, using enough projectives of $(\class{C},\tilclass{F})$ corresponds to the cofibrant replacement functor which we denote by $Q$. However, by using enough projectives of $(\tilclass{C},\class{F})$ we also get a functor which we denote by $\widetilde{Q}$ and by using enough injectives of $(\class{C},\tilclass{F})$ we get a functor we denote by $\widetilde{R}$. When we encounter multiple abelian model structures we use subscripts such as $\class{M}_1,\class{M}_2,\class{M}_3$ and denote these associated functors with notations such as $R_2$, $\widetilde{Q}_1$, $\widetilde{R}_3$ etc.

Finally, by the \emph{core} of an abelian model structure $\class{M} = (\class{C},\class{W},\class{F})$ we mean the class $\class{C} \cap \class{W} \cap \class{F}$. A recent result appearing in~\cite{gillespie-hovey triples} will prove fundamental to this paper. It says that whenever $(\tilclass{C},\class{F})$ and $(\class{C},\tilclass{F})$ are complete hereditary cotorsion pairs with equal cores and $\tilclass{F} \subseteq \class{F}$, then there is a unique thick class $\class{W}$ yielding a Hovey triple $\class{M} = (\class{C},\class{W},\class{F})$ with $\class{C} \cap \class{W} = \tilclass{C}$ and $\class{W} \cap \class{F} = \tilclass{F}$. Besides~\cite{hovey} we will refer to~\cite{hovey-model-categories} for any other basics from the theory of model categories. 

\subsection{Recollement situations}

The homotopy category of an abelian model category is always a triangulated category~\cite[Section~7]{hovey-model-categories}. Quillen functors, which are left adjoint functors between abelian model categories preserving the cofibrant and trivially cofibrant objects, naturally give rise to \emph{left derived functors} between these triangulated homotopy categories. This paper aims to describe recollements of triangulated categories using Quillen functors. Loosely, a recollement is an ``attachment'' of two triangulated categories. The standard reference is~\cite{BBD-perverse sheaves}, although the definitions below will surely suffice for the reader.

\begin{definition}\label{def-localization sequence}
Let $\class{T}' \xrightarrow{F} \class{T} \xrightarrow{G} \class{T}''$ be a sequence of exact functors between triangulated categories. We say it is a \emph{localization sequence} when there exists right adjoints $F_{\rho}$ and $G_{\rho}$ giving a diagram of functors as below with the listed properties.
$$\begin{tikzcd}
\class{T}'
\rar[to-,
to path={
([yshift=0.5ex]\tikztotarget.west) --
([yshift=0.5ex]\tikztostart.east) \tikztonodes}][swap]{F}
\rar[to-,
to path={
([yshift=-0.5ex]\tikztostart.east) --
([yshift=-0.5ex]\tikztotarget.west) \tikztonodes}][swap]{F_{\rho}}
& \class{T}
\rar[to-,
to path={
([yshift=0.5ex]\tikztotarget.west) --
([yshift=0.5ex]\tikztostart.east) \tikztonodes}][swap]{G}
\rar[to-,
to path={
([yshift=-0.5ex]\tikztostart.east) --
([yshift=-0.5ex]\tikztotarget.west) \tikztonodes}][swap]{G_{\rho}}
& \class{T}'' \\
\end{tikzcd}$$
\begin{enumerate}
\item The right adjoint $F_{\rho}$ of $F$ satisfies $F_{\rho} \circ F \cong \text{id}_{\class{T}'}$.
\item The right adjoint $G_{\rho}$ of $G$ satisfies $G \circ G_{\rho} \cong \text{id}_{\class{T}''}$.
\item For any object $X \in \class{T}$, we have $GX = 0$ iff $X \cong FX'$ for some $X' \in \class{T}'$.
\end{enumerate}
The notion of a \emph{colocalization sequence} is the dual. That is, there must exist left adjoints $F_{\lambda}$ and $G_{\lambda}$ with the analogous properties.
\end{definition}

Note the similarity in the definitions above to the notion of a split exact sequence, but for adjunctions. It is true that if  $\class{T}' \xrightarrow{F} \class{T} \xrightarrow{G} \class{T}''$ is a localization sequence then  $\class{T}'' \xrightarrow{G_{\rho}} \class{T} \xrightarrow{F_{\rho}} \class{T}'$ is a colocalization sequence and if  $\class{T}' \xrightarrow{F} \class{T} \xrightarrow{G} \class{T}''$ is a colocalization sequence then  $\class{T}'' \xrightarrow{G_{\lambda}} \class{T} \xrightarrow{F_{\lambda}} \class{T}'$ is a localization sequence. This brings us to the definition of a recollement where the sequence of functors  $\class{T}' \xrightarrow{F} \class{T} \xrightarrow{G} \class{T}''$ is both a localization sequence and a colocalization sequence.

\begin{definition}\label{def-recollement}
Let $\class{T}' \xrightarrow{F} \class{T} \xrightarrow{G} \class{T}''$ be a sequence of exact functors between triangulated categories. We say $\class{T}' \xrightarrow{F} \class{T} \xrightarrow{G} \class{T}''$ induces a \emph{recollement} if it is both a localization sequence and a colocalization sequence as shown in the picture.
\[
\xy
(-20,0)*+{\class{T}'};
(0,0)*+{\class{T}};
(20,0)*+{\class{T}''};
{(-18,0) \ar^{F} (-2,0)};
{(-2,0) \ar@/^1pc/@<0.5em>^{F_{\rho}} (-18,0)};
{(-2,0) \ar@/_1pc/@<-0.5em>_{F_{\lambda}} (-18,0)};
{(2,0) \ar^{G} (18,0)};
{(18,0) \ar@/^1pc/@<0.5em>^{G_{\rho}} (2,0)};
{(18,0) \ar@/_1pc/@<-0.5em>_{G_{\lambda}} (2,0)};
\endxy
\]
\end{definition}
So the idea is that a recollement is a colocalization sequence ``glued'' with a localization sequence.

\subsection{Deconstructibility}\label{subsec-deconstruct}

In applications the abelian categories we work with are actually Grothendieck categories and in the final Section~\ref{subsec-generalizations} we will use the notion of deconstructibility in Grothendieck categories. This concept has been nicely developed by {\v{S}}\v{t}ov{\'{\i}}{\v{c}}ek in~\cite{stovicek}. 

\begin{definition}\label{def-deconstructible}
A class of objects $\class{F}$ is called \emph{deconstructible} if $\class{F} = \textnormal{Filt-}\class{S}$ for some set $\class{S}$. Here  $\textnormal{Filt-}\class{S}$ denotes the class of all transfinite extensions of $\class{S}$. So in particular, $\class{S} \subseteq \class{F}$ and $\class{F}$ is the class of all transfinite extensions of $\class{S}$.
\end{definition}

By a \emph{transfinite extension of $\class{S}$} we mean a colimit
$\colim X_{\alpha}$, where $X \mathcolon \lambda \xrightarrow{} \cat{G}$ is a
colimit-preserving functor whose domain is an ordinal $\lambda$ and satisfying the following: $X_0 = 0$ and each map $X_{\alpha} \xrightarrow{}
X_{\alpha + 1}$ is a monomorphism with cokernel in $\class{S}$.
 
We then have the following very useful proposition. 

\begin{proposition}\label{prop-cotorsion pairs and deconstructibility}
Let $\cat{G}$ be a Grothendieck category and $\class{X}$ a class of objects. Set $\class{Y} = \rightperp{\class{X}}$. Then the following are equivalent:
\begin{enumerate}
\item $\class{X}$ is deconstructible, closed under retracts, and contains a generating set $\{U_i\}$ for $\cat{G}$.
\item $(\class{X},\class{Y})$ is a complete cotorsion pair cogenerated by some set $\class{S}$.
\item $(\class{X},\class{Y})$ is a cotorsion pair cogenerated by some set $\class{S}$ and $\class{X}$ contains a generating set $\{U_i\}$ for $\cat{G}$.
\end{enumerate}
Moreover, in case (3), every object in $\class{X}$ is a retract of a transfinite extension of objects in $\class{S} \cup \{U_i\}$. 
\end{proposition}

\begin{proof}
We first show (1) implies (2). So assume we have some set $\class{S}'$ such that $\class{X} = \textnormal{Filt-}\class{S}'$ is deconstructible. We note that by~\cite[Lemma~1.6]{stovicek} $\class{X}$ itself must be closed under transfinite extensions.
Set $\class{S} = \class{S}' \cup \{U_i\}$. Then by~\cite[Corollary~2.14~(2)]{saorin-stovicek} $\class{S}$ cogenerates a complete cotorsion pair $(\leftperp{(\rightperp{\class{S}})},\rightperp{\class{S}})$ and $\leftperp{(\rightperp{\class{S}})}$ consists precisely of retracts of transfinite extensions of objects in $\class{S}$. But since $\class{X}$ is closed under transfinite extensions, retracts, and contains $\{U_i\}$, it follows that $\leftperp{(\rightperp{\class{S}})}=\class{X}$.

For the converse, suppose $(\class{X},\class{Y})$ is a complete cotorsion pair cogenerated by a set $\class{S}$, meaning $\rightperp{\class{S}} = \class{Y}$. If $\{U_i\}$ is any generating set for $\cat{G}$ and if for each $i$ we have an epimorphism $X_i \xrightarrow{} U_i$ with $X_i \in \class{X}$, then $\{X_i\}$ too is a generating set. So since the cotorsion pair is complete we can always find a generating set $\{U_i\}$ contained in $\class{X}$. 
It is easy to see that $\class{X}$ is closed under retracts, so it is left to show it is deconstructible. We can't expect to have $\class{X} = \textnormal{Filt-}\class{S}$; we must replace $\class{S}$ with another set $\class{S}'$. We start by considering $\class{T} = \class{S} \cup \{U_i\}$. Then still $\rightperp{\class{T}} = \class{Y}$ because $\{U_i\} \subseteq \class{X}$. So now $(\leftperp{(\rightperp{\class{T}})},\rightperp{\class{T}}) = (\leftperp{(\rightperp{\class{S}})},\rightperp{\class{S}}) = (\class{X},\class{Y})$ and again by~\cite[Corollary~2.14~(2)]{saorin-stovicek} we conclude that $\class{X}$ consists precisely of retracts of objects in the deconstructible class $\textnormal{Filt-}\class{T}$. But now~\cite[Proposition~2.9~(1)]{stovicek} tells us that $\class{X} = \textnormal{Filt-}\class{S}'$ for some set $\class{S}'$. So (2) implies (1), and the above arguments also make it clear that (2) and (3) are equivalent. 
\end{proof}

\section{Localization and recollement via model categories}\label{sec-recollement}

Our main goal is to prove Theorems~\ref{them-right recollement} and~\ref{them-left recollement} which automatically produce a recollement from three interrelated Hovey triples. 

\begin{proposition}\label{prop-right localization}
Let $\class{M}_1 = (\class{C},\class{W}_1,\class{F}_1)$ and $\class{M}_2 = (\class{C},\class{W}_2,\class{F}_2)$ be hereditary abelian model structures with equal cores
$\class{C} \cap \class{W}_1 \cap \class{F}_1 = \class{C} \cap \class{W}_2 \cap \class{F}_2$, and $\class{F}_2 \subseteq \class{F}_1$. Then there exists an hereditary abelian model structure $\class{M}_1/\class{M}_2 = (\widetilde{\class{C}}_2, \class{W},\class{F}_1)$, called the \textbf{right localization} of $\class{M}_1$ with respect to $\class{M}_2$. Here
\begin{align*}
   \class{W}  &= \{\, X \in \class{A} \, | \, \exists \, \text{a short exact sequence } \, X \rightarrowtail F_2 \twoheadrightarrow C_1 \, \text{ with} \, F_2 \in \class{F}_2 \, , C_1 \in \tilclass{C}_1 \,\} \\
           &= \{\, X \in \class{A} \, | \, \exists \, \text{a short exact sequence } \, F'_2 \rightarrowtail C'_1 \twoheadrightarrow X \, \text{ with} \, F'_2 \in \class{F}_2 \, , C'_1 \in \tilclass{C}_1 \,\}.
          \end{align*}
\end{proposition}

\begin{proof}
We have the two complete hereditary cotorsion pairs $(\widetilde{\class{C}}_1,\class{F}_1)$ and $(\widetilde{\class{C}}_2, \class{F}_2)$ satisfying
$\class{F}_2 \subseteq \class{F}_1$ and $\widetilde{\class{C}}_1 \cap \class{F}_1 = \widetilde{\class{C}}_2 \cap \class{F}_2$. Applying~\cite[Theorem~1.1]{gillespie-hovey triples} we immediately obtain a unique thick class $\class{W}$, with the two descriptions above, making $(\widetilde{\class{C}}_2, \class{W},\class{F}_1)$ an hereditary abelian model structure.
\end{proof}

We note that the above Proposition is a generalization of a recent result of Becker from~\cite{becker}. The same is true for the next proposition. For the statement, recall from Section~\ref{subsec-abelian model cats} the notation we are using for the functors $R$, $\widetilde{R}$, $Q$, and $\widetilde{Q}$ associated to an abelian model structure $\class{M}$. 

\begin{proposition}\label{prop-colocalization sequence}
Let $\class{M}_1 = (\class{C},\class{W}_1,\class{F}_1)$ and $\class{M}_2 = (\class{C},\class{W}_2,\class{F}_2)$ be hereditary abelian model structures with
equal cores and $\class{F}_2 \subseteq \class{F}_1$. Then we have a colocalization sequence $\textnormal{Ho}(\class{M}_2)  \xrightarrow{R(\textnormal{Id}) = R_2} \textnormal{Ho}(\class{M}_1) \xrightarrow{R(\textnormal{Id}) = R_1} \textnormal{Ho}(\class{M}_1/\class{M}_2)$ with left adjoints  $Q_1$ and $\widetilde{Q}_2$ as shown below. Here the vertical arrows are the standard equivalences passing between the homotopy categories and the equivalent homotopy categories of fibrant-cofibrant objects.
\[
\begin{tikzpicture}[node distance=3.5 cm, auto]
\node (A)  {$\textnormal{Ho}(\class{M}_2)$};
\node (D) [below of=A] {$(\class{C} \cap \class{F}_2)/\sim$};
\node (B) [right of=A] {$\textnormal{Ho}(\class{M}_1)$};
\node (C) [right of=B] {$\textnormal{Ho}(\class{M}_1/\class{M}_2)$};
\node (E) [right of=D] {$(\class{C} \cap \class{F}_1)/\sim$};
\node (F) [right of=E] {$(\widetilde{\class{C}}_2 \cap \class{F}_1)/\sim$};

%
%
\draw[<-] (A.10) to node {$Q_1$} (B.170);
\draw[->] (A.350) to node [swap] {$R_2$} (B.190);
\draw[<-] (D.7) to node {$R_2$} (E.173);
\draw[->] (D.353) to node [swap] {\textnormal{inc}} (E.187);
\draw[<-] (E.7) to node {\textnormal{inc}} (F.173);
\draw[->] (E.353) to node [swap] {$\widetilde{Q}_2$} (F.187);
\draw[<-] (B.10) to node {$\widetilde{Q}_2$} (C.173);
\draw[->] (B.350) to node [swap] {$R_1$} (C.187);
%
%
\draw[<-] (B.290) to node {\textnormal{inc}} (E.70);
\draw[->] (B.250) to node [swap] {$Q_1 \circ R_1$} (E.110);
\draw[<-] (A.290) to node {\textnormal{inc}} (D.70);
\draw[->] (A.250) to node [swap] {$Q_2 \circ R_2$} (D.110);
\draw[<-] (C.290) to node {\textnormal{inc}} (F.70);
\draw[->] (C.250) to node [swap] {$\widetilde{Q}_2 \circ R_1$} (F.110);
\end{tikzpicture}
\]
\end{proposition}

Above we followed a convention that we will use throughout the paper whenever we write functor diagrams: No matter which direction functors are going, we will always write left adjoints on the top or the left, while the right adjoints will be written on the bottom or right.

\begin{proof}
The identity functors $\class{M}_2  \xrightarrow{\textnormal{Id}} \class{M}_1 \xrightarrow{\textnormal{Id}} \class{M}_1/\class{M}_2$ are right Quillen since they preserve fibrant and trivially fibrant objects. So we get the adjunctions as shown below, where using the definitions of left and right derived functors, we have computed $(L(\textnormal{Id}), R(\textnormal{Id})) = (Q_1,R_2)$ and $(L(\textnormal{Id}), R(\textnormal{Id})) = (\widetilde{Q}_2,R_1)$.
\[
\begin{tikzpicture}[node distance=3.5 cm, auto]
\node (A)  {$\textnormal{Ho}(\class{M}_2)$};
\node (B) [right of=A] {$\textnormal{Ho}(\class{M}_1)$};
\node (C) [right of=B] {$\textnormal{Ho}(\class{M}_1/\class{M}_2)$};
%
%
\draw[<-] (A.10) to node {$Q_1$} (B.170);
\draw[->] (A.350) to node [swap] {$R_2$} (B.190);
\draw[<-] (B.10) to node {$\widetilde{Q}_2$} (C.173);
\draw[->] (B.350) to node [swap] {$R_1$} (C.187);
\end{tikzpicture}
\]
These functors are exact in the sense that they send triangles to triangles. So to show we have a colocalization sequence, it remains to show
\begin{enumerate}
\item $Q_1 \circ R_2 \cong 1_{\textnormal{Ho}(\class{M}_2)}$.

\item $R_1 \circ \widetilde{Q}_2 \cong 1_{\textnormal{Ho}(\class{M}_1/\class{M}_2)}$.

\item The essential image of $R_2$ equals the kernel of $R_1$.
\end{enumerate}
To prove (1), let $f : A \xrightarrow{} B$ be a morphism in $\cat{A}$. The functor $R_2$ acts by $f \mapsto \hat{f}$ where $\hat{f}$ is any map making the diagram below commute. Here the rows are exact, $F_2,F'_2 \in \class{F}_2$ and $C_2,C'_2 \in  \widetilde{\class{C}}_2 = \class{C} \cap \class{W}_2$.
$$\begin{CD}
0 @>>> A       @>j_A>>  F_2 @>>>     C_2 @>>> 0\\
@. @VVfV            @VV \hat{f} V     @.  \\
0 @>>> B       @>j_B>>       F'_2  @>>>     C'_2  @>>> 0 \\
\end{CD}$$
Now applying $Q_1$ to $\hat{f}$ gives us $\bar{f}$ in the next commutative diagram. Here $F_1,F'_1 \in \widetilde{\class{F}}_1 = \class{W}_1 \cap \class{F}_1$ and $C,C' \in \class{C}$.
$$\begin{CD}
0 @>>> F_1       @>>>  C @>q_A>>     F_2 @>>> 0\\
@. @.          @V\bar{f}VV     @V\hat{f}VV  \\
0 @>>> F'_1       @>>>       C'  @>q_B>>     F'_2  @>>> 0 \\
\end{CD}$$
By the argument in the last paragraph of the proof of Lemma~\ref{lemma-quotient map} we see that $\class{W}_1 \subseteq \class{W}_2$. Thus $j_A, j_B, q_A, q_B$ are all weak equivalences in $\class{M}_2$. So, in $\textnormal{Ho}(\class{M}_2)$, we have a commutative diagram
$$\begin{CD}
 C       @>q_A>> F_2  @<j_A<<   A\\
@V\bar{f}VV     @V\hat{f}VV  @VfVV \\
 C'         @>q_B>> F'_2 @<j_B<<  B \\
\end{CD}$$
giving rise to a natural isomorphism $\{j^{-1}_A \circ q_A\} : Q_1 \circ R_2 \cong 1_{\textnormal{Ho}(\class{M}_2)}$.

A similar type of argument to the above will prove (2).

For (3), we start by claiming $\ker{R_1} = \class{W}$, where $\class{W}$ is the class of trivial objects in $\class{M}_1/\class{M}_2$; so they are precisely the zero objects in $\textnormal{Ho}(\class{M}_1/\class{M}_2)$. ($\subseteq$) Say $R_1X \in \class{W}$. This means that the fibrant replacement $R_1X = F_1$ obtained by finding a short exact sequence $0 \xrightarrow{} X \xrightarrow{} F_1 \xrightarrow{} C_1 \xrightarrow{} 0$ with $F_1 \in \class{F}_1$ and $C_1 \in \tilclass{C}_1$, actually has $F_1 \in \class{W} \cap \class{F}_1 = \class{F}_2$. So by the definition of the class $\class{W}$ in the statement of Proposition~\ref{prop-right localization}, we have $X \in \class{W}$. ($\supseteq$) Let $X \in \class{W}$. Then computing $R_1X$ with a short exact sequence as we just did, we conclude $R_1X \in \class{F}_2 = \class{W} \cap \class{F}_1$. In particular, $R_1X = 0$ in $\textnormal{Ho}(\class{M}_1/\class{M}_2)$.

 We proceed to finish the proof of (3). It is clear that the literal image of $R_2 : \textnormal{Ho}(\class{M}_2) \xrightarrow{} \textnormal{Ho}(\class{M}_1)$ is the class $\class{F}_2 = \class{W} \cap \class{F}_1$, which is contained in $\ker{R_1}$. By definition, the essential image of $R_2$ is therefore the class of all $X$ that are isomorphic, in $\textnormal{Ho}(\class{M}_1)$, to an object of $\class{F}_2$. But it is easy to see that the kernel of any additive functor is closed under isomorphic closure. So $\ker{R_1}$ must contain the essential image of $R_2$. On the other hand, if $W \in \class{W} = \ker{R_1}$, we have a short exact sequence $0 \xrightarrow{} W \xrightarrow{} F_2 \xrightarrow{} C_1 \xrightarrow{} 0$ with $F_2 \in \class{F}_2$ and $C_1 \in \tilclass{C}_1 = \tilclass{C}_2 \cap \class{W}$. So $W \xrightarrow{} F_2$ is a weak equivalence in 
$\class{M}_1/\class{M}_2$ and we are done. 
\end{proof}

For an abelian model structure $\class{M} = (\class{C},\class{W},\class{F})$ on $\cat{A}$, we let $\gamma :\cat{A} \xrightarrow{} \textnormal{Ho}(\cat{M})$ denote the canonical localization functor.

\begin{lemma}\label{lemma-quotient map}
Whenever we have two abelian model structures $\class{M} = (\class{C},\class{W},\class{F})$ and $\class{M}' = (\class{C}',\class{W}',\class{F}')$ on $\cat{A}$
with $\class{W} \subseteq \class{W}'$, then we have a canonical \textbf{\emph{quotient functor}} $\bar{\gamma} : \textnormal{Ho}(\cat{M}) \xrightarrow{} \textnormal{Ho}(\cat{M}')$ for which $\gamma' = \bar{\gamma} \circ \gamma$.

In particular, if $\class{M} = (\class{C},\class{W},\class{F})$ and $\class{M}' = (\class{C},\class{W}',\class{F}')$ have the same cofibrant objects and $\class{F}' \subseteq \class{F}$, then we have $\class{W} \subseteq \class{W}'$. So we have $\bar{\gamma} : \textnormal{Ho}(\cat{M}) \xrightarrow{} \textnormal{Ho}(\cat{M}')$.
\end{lemma}

\begin{proof}
In any abelian model structure a map is a weak equivalence if and only if it factors as a monomorphism with trivial cokernel followed by an epimorphism with trivial kernel.
So if $\class{W} \subseteq \class{W}'$, then the localization functor $\gamma' : \cat{A} \xrightarrow{} \textnormal{Ho}(\cat{M}')$ sends weak equivalences in $\class{M}$ to isomorphisms. So the universal property of $\gamma$ guarantees the unique functor  $\bar{\gamma} : \textnormal{Ho}(\cat{M}) \xrightarrow{} \textnormal{Ho}(\cat{M}')$ for which $\gamma' = \bar{\gamma} \circ \gamma$.

In particular, say $\class{M} = (\class{C},\class{W},\class{F})$ and $\class{M}' = (\class{C},\class{W}',\class{F}')$ satisfy $\class{F}' \subseteq \class{F}$. Then we have the containment of trivially cofibrant objects $\class{C} \cap \class{W} \subseteq \class{C} \cap \class{W}'$ and equality of trivially fibrant objects $\class{W} \cap \class{F} = \class{W} \cap \class{F}'$. From this, one can use the factorization and two of three property to see that $\class{W} \subseteq \class{W}'$.
\end{proof}

\begin{theorem}[Right recollement theorem]\label{them-right recollement}
Let $\cat{A}$ be an abelian category with three hereditary model structures, as below, whose cores all coincide: $$\class{M}_1 = (\class{C}, \class{W}_1, \class{F}_1) , \ \ \ \class{M}_2 = (\class{C}, \class{W}_2, \class{F}_2) , \ \ \ \class{M}_3 = (\class{C}, \class{W}_3, \class{F}_3).$$ If $\class{W}_3 \cap \class{F}_1 = \class{F}_2$ and $\class{F}_3 \subseteq  \class{F}_1$ (or
equivalently, $\widetilde{\class{C}}_2 \cap \class{W}_3 = \widetilde{\class{C}}_1$ and $\class{F}_2 \subseteq \class{W}_3$), then $\textnormal{Ho}(\class{M}_1/\class{M}_2) = \textnormal{Ho}(\class{M}_3)$ and $\textnormal{Ho}(\class{M}_1/\class{M}_3) \cong \textnormal{Ho}(\class{M}_2)$. In fact, $\class{M}_1/\class{M}_3$ is Quillen equivalent to $\class{M}_2$, while $\class{M}_1/\class{M}_2$ is Quillen equivalent to $\class{M}_3$, and we even have a recollement as shown below.
\[
\begin{tikzpicture}[node distance=3.5cm]
\node (A) {$\textnormal{Ho}(\class{M}_2)$};
\node (B) [right of=A] {$\textnormal{Ho}(\class{M}_1)$};
\node (C) [right of=B] {$\textnormal{Ho}(\class{M}_3)$};
\draw[<-,bend left=40] (A.20) to node[above]{$Q_1$} (B.160);
\draw[->] (A) to node[above]{\small $R_2$} (B);
\draw[<-,bend right=40] (A.340) to node [below]{$\widetilde{Q}_3 \circ R_1$} (B.200);

\draw[<-,bend left] (B.20) to node[above]{\small $\lambda = \widetilde{Q}_2$} (C.160);
\draw[->] (B) to node[above]{$\bar{\gamma}$} (C);
\draw[<-,bend right] (B.340) to node [below]{\small $\rho = R_3$} (C.200);
\end{tikzpicture}\]
Here, the functor $\bar{\gamma}$ is the quotient functor of Lemma~\ref{lemma-quotient map}.
\end{theorem}

\begin{proof}
First we show that the two conditions are equivalent. That is, $\class{W}_3 \cap \class{F}_1 = \class{F}_2$  and $\class{F}_3 \subseteq  \class{F}_1$ if and only if $\widetilde{\class{C}}_2 \cap \class{W}_3 = \widetilde{\class{C}}_1$ and $\class{F}_2 \subseteq \class{W}_3$. For the ``only if'' part, the only part that is not clear is $\widetilde{\class{C}}_2 \cap \class{W}_3 \subseteq \widetilde{\class{C}}_1$. So assume $C \in \widetilde{\class{C}}_2 \cap \class{W}_3$. Use enough projectives of $(\widetilde{\class{C}}_1, \class{F}_1)$ to find a short exact sequence $0 \xrightarrow{} F_1 \xrightarrow{} C_1 \xrightarrow{} C \xrightarrow{} 0$ with $F_1 \in \class{F}_1$, and $C_1 \in \widetilde{\class{C}}_1 \subseteq \class{W}_3$. Since $\class{W}_3$ is thick we see that $F_1 \in \class{W}_3 \cap \class{F}_1 = \class{F}_2$. So with $F_1 \in \class{F}_2$ and $C \in \widetilde{\class{C}}_2$, the short exact sequence must split, making $C$ a retract of $C_1$. Hence $C$ must be in $\widetilde{\class{C}}_1$ finishing the proof of the ``only if'' part. For the ``if'' part, the analogous part to show is $\class{W}_3 \cap \class{F}_1 \subseteq \class{F}_2$. This follows by a similar argument: For $W \in \class{W}_3 \cap \class{F}_1$ start by finding a short exact sequence $0 \xrightarrow{} W \xrightarrow{} F_2 \xrightarrow{} C_2 \xrightarrow{} 0$ with $F_2 \in \class{F}_2$, and $C_2 \in \widetilde{\class{C}}_2$. Since $\class{W}_3$ is thick we get $C_2 \in \widetilde{\class{C}}_2 \cap \class{W}_3 = \widetilde{\class{C}}_1$ and the sequence splits.

Having shown that the two conditions are equivalent, we see at once that they make $(\widetilde{\class{C}}_2, \class{W}_3, \class{F}_1)$ into a Hovey triple. By the uniqueness of the thick class in a Hovey triple (see~\cite[Proposition~3.2]{gillespie-recollements}) we conclude that $\class{M}_1/\class{M}_2 = (\widetilde{\class{C}}_2, \class{W}_3, \class{F}_1)$. It is then easy to see that the identity functor from $\class{M}_3$ to $\class{M}_1/\class{M}_2$ is a right Quillen functor. But since $\class{M}_3$ and $\class{M}_1/\class{M}_2$ both have $\class{W}_3$ as its class of trivial objects it follows (see~\cite[Lemma~5.8]{hovey} or~\cite[Lemma~2.7]{gillespie-recollements}) that the model structures have the same weak equivalences. So not only is the identity adjunction a Quillen equivalence but we have the equality $\textnormal{Ho}(\class{M}_3) = \textnormal{Ho}(\class{M}_1/\class{M}_2)$.

Note next that $\textnormal{Ho}(\class{M}_2) \cong \textnormal{Ho}(\class{M}_1/\class{M}_3)$. The reason is that $$\class{M}_1/\class{M}_3 = (\widetilde{\class{C}}_3, \class{V}, \class{F}_1)$$ for some thick class $\class{V}$, and so
$$\textnormal{Ho}(\class{M}_1/\class{M}_3)\cong (\widetilde{\class{C}}_3 \cap \class{F}_1)/\sim \ =  (\class{C} \cap \class{F}_2)/\sim  \ \cong \textnormal{Ho}(\class{M}_2).$$
The homotopy relations $\sim$ agree in the two model structures by~\cite[Proposition~4.4~(5)]{gillespie-exact model structures} since the cores of  $\class{M}_1/\class{M}_3 = (\widetilde{\class{C}}_3, \class{V}, \class{F}_1)$ and $\class{M}_2 = (\class{C}, \class{W}_2, \class{F}_2)$ coincide. We claim that the identity map $\class{M}_1/\class{M}_3 \xrightarrow{} \class{M}_2$ is a (left) Quillen equivalence. First note that it preserves cofibrant and trivially cofibrant objects, and this implies it is a left Quillen functor. To see it is a Quillen equivalence, we must by~\cite[Definition~1.3.12]{hovey-model-categories} show that for all cofibrant $A \in  \class{M}_1/\class{M}_3$ and fibrant $B \in \class{M}_2$ argue that a map $f : A \xrightarrow{} B$ is a weak equivalence in $\class{M}_2$ if and only if it is a weak equivalence in $\class{M}_1/\class{M}_3$. So let $f : A \xrightarrow{} B$ be a map with $A \in \widetilde{\class{C}}_3$ and $B \in \class{F}_2$. Assuming it is a weak equivalence in $\class{M}_2$, we may factor it as $f = pi$ where $i : A \xrightarrow{} C$ has $\cok{i} \in \widetilde{\class{C}}_2$ and $p : C \xrightarrow{} B$ has $\ker{p} \in \class{W}_2 \cap \class{F}_2 = \class{W}_3 \cap \class{F}_3$. Note then that $\ker{p} \in \class{F}_3 = \class{V} \cap \class{F}_1$ is automatically trivially fibrant in $\class{M}_1/\class{M}_3$. But since $\class{W}_3$ is thick we also have $C \in \class{W}_3$, and thus $\cok{i} \in \class{W}_3$. So $\cok{i} \in \widetilde{\class{C}}_2 \cap \class{W}_3 = \widetilde{\class{C}}_1 = \widetilde{\class{C}}_3 \cap \class{V}$. That is, $\cok{i}$ is trivially cofibrant in $\class{M}_1/\class{M}_3$. Since we have shown $f$ factors as a trivial cofibration followed by a trivial fibration in $\class{M}_1/\class{M}_3$ we conclude it is also a weak equivalence in $\class{M}_1/\class{M}_3$. Conversely, if $f$ is a weak equivalence in $\class{M}_1/\class{M}_3$, the argument reverses and we see that $f$ is a weak equivalence in $\class{M}_2$.

Now we wish to construct the recollement.
Using that $\class{F}_2 \subseteq \class{F}_1$, we apply Proposition~\ref{prop-colocalization sequence} to obtain a colocalization sequence which appears as the top horizontal row in the diagram below. However, we also have $\class{F}_3 \subseteq \class{F}_1$ and so we have the analogous colocalization sequence involving $\class{M}_3$. But we may rewrite
this last colocalization ``backwards'' so that it appears as a localization sequence. This is the bottom row of the following diagram which also uses the equality $\textnormal{Ho}(\class{M}_3) = \textnormal{Ho}(\class{M}_1/\class{M}_2)$ established above:
\[
\tag{\text{$***$}}
\begin{tikzpicture}[node distance=3.5 cm, auto]
\node (A)  {$\textnormal{Ho}(\class{M}_2)$};
\node (D) [below of=A] {$\textnormal{Ho}(\class{M}_1/\class{M}_3)$};
\node (B) [right of=A] {$\textnormal{Ho}(\class{M}_1)$};
\node (C) [right of=B] {$\textnormal{Ho}(\class{M}_1/\class{M}_2)$};
\node (E) [right of=D] {$\textnormal{Ho}(\class{M}_1)$};
\node (F) [right of=E] {$\textnormal{Ho}(\class{M}_3)$};

%
%
\draw[<-] (A.10) to node {$Q_1$} (B.170);
\draw[->] (A.350) to node [swap] {$R_2$} (B.190);
\draw[->] (D.6) to node {$\widetilde{Q}_3$} (E.170);
\draw[<-] (D.353) to node [swap] {$R_1$} (E.190);
\draw[->] (E.10) to node {$Q_1$} (F.170);
\draw[<-] (E.350) to node [swap] {$R_3$} (F.190);
\draw[<-] (B.10) to node {$\widetilde{Q}_2$} (C.173);
\draw[->] (B.350) to node [swap] {$R_1$} (C.187);
%
%
\draw[-] (B.280) to node {} (E.80);
\draw[-] (B.260) to node [swap] {} (E.100);
\draw[->] (A.290) to node {$R_2$} (D.70);
\draw[<-] (A.250) to node [swap] {$\widetilde{Q}_3$} (D.110);
\draw[-] (C.280) to node {} (F.80);
\draw[-] (C.260) to node [swap] {} (F.100);
\end{tikzpicture}
\]
Again, all top functors are left adjoints and all bottom functors are right adjoints. We now let $\bar{\gamma} : \textnormal{Ho}(\cat{M}_1) \xrightarrow{} \textnormal{Ho}(\cat{M}_3)  = \textnormal{Ho}(\class{M}_1/\class{M}_2)$ be the quotient functor of Lemma~\ref{lemma-quotient map}. We claim, that as functors $\textnormal{Ho}(\cat{M}_1)$ to $\textnormal{Ho}(\class{M}_1/\class{M}_2)$, there is a natural isomorphism of functors $\{\,j_A\,\} : \bar{\gamma} \, \cong R_1$.
Indeed given $f : A \xrightarrow{} B$, recall the functor $R_1$ works by applying enough injectives using the cotorsion pair $(\widetilde{\class{C}}_1,\class{F}_1)$. That is, $R_1$ acts by $f \mapsto \hat{f}$ where $\hat{f}$ is any map making the diagram below commute. Here, the rows are exact and $F_1,F'_1 \in \class{F}_1$ and $C_1,C'_1 \in  \widetilde{\class{C}}_1$.
$$\begin{CD}
0 @>>> A       @>j_A>>  F_1 @>>>     C_1 @>>> 0\\
@. @VVfV            @VV \hat{f} V     @.  \\
0 @>>> B       @>j_B>>       F'_1  @>>>     C'_1  @>>> 0 \\
\end{CD}$$
Since $\widetilde{\class{C}}_1 = \class{C} \cap \class{W}_1 = \leftperp{F}_1 = \widetilde{\class{C}}_2 \cap \class{W}_3$, we see in particular that $C_1,C'_1 \in \class{W}_3$. That is, they are trivial in $\class{M}_1/\class{M}_2$ which means the maps $\{\,j_A\,\}$ are providing a natural isomorphism $\{\,j_A\,\} : \bar{\gamma} \, \cong R_1$.

Similarly we get that the functors $\bar{\gamma}, Q_1 : \textnormal{Ho}(\cat{M}_1) \xrightarrow{} \textnormal{Ho}(\class{M}_3)$ admit a natural isomorphism $\{\,p_A\,\} : Q_1 \cong \, \bar{\gamma}$. To see this, recall the functor $Q_1$ works by using enough projectives of the cotorsion pair $(\class{C}, \widetilde{\class{F}}_1)$. So for a given $f : A \xrightarrow{} B$ we construct a diagram as below with exact rows and $F_1,F'_1 \in \widetilde{\class{F}}_1$ and $C_1,C'_1 \in  \class{C}$.
$$\begin{CD}
0 @>>> F_1       @>>>  C_1 @>p_A>>     A @>>> 0\\
@. @.          @V\hat{f}VV     @VfVV  \\
0 @>>> F'_1       @>>>       C'_1  @>p_B>>     B  @>>> 0 \\
\end{CD}$$
But $F_1,F'_1 \in \widetilde{\class{F}}_1 = \class{W}_1 \cap \class{F}_1 = \rightperp{C} =  \class{W}_3 \cap \class{F}_3$, so again, $F_1,F'_1$ are trivial in $\class{M}_3$. This show the maps $\{\,p_A\,\}$ are providing a natural isomorphism $\{\,p_A\,\} : Q_1 \cong \,\bar{\gamma}$.

Going back to $(***)$ one can easily check that using the natural isomorphism $\{\,j_A\,\} : \bar{\gamma} \, \cong R_1$, we may simply replace $R_1$ by $\bar{\gamma}$ and the top row will still be a colocalization sequence. Similarly, using the natural isomorphism $\{\,p_A\,\} : Q_1 \cong \,\bar{\gamma}$ we may replace $Q_1$ with $\bar{\gamma}$ and the bottom row will remain a localization sequence. That is, the right square of $(***)$ is glued along $\bar{\gamma}$, and becoming
\[
\begin{tikzpicture}[node distance=3.5cm]
\node (B) [right of=A] {$\textnormal{Ho}(\class{M}_1)$};
\node (C) [right of=B] {$\textnormal{Ho}(\class{M}_3)$};
\draw[<-,bend left] (B.20) to node[above]{\small $\lambda = \widetilde{Q}_2$} (C.160);
\draw[->] (B) to node[above]{$\bar{\gamma}$} (C);
\draw[<-,bend right] (B.340) to node [below]{\small $\rho = R_3$} (C.200);
\end{tikzpicture}
\]

We now turn to the left square of the diagram $(***)$. We claim the functor $\textnormal{Ho}(\class{M}_2) \xrightarrow{R_2} \textnormal{Ho}(\class{M}_1)$ is isomorphic to
the composite $\textnormal{Ho}(\class{M}_2) \xrightarrow{R_2} \textnormal{Ho}(\class{M}_1/\class{M}_3)  \xrightarrow{\widetilde{Q}_3} \textnormal{Ho}(\class{M}_1)$. Indeed, $R_2$ acts by $f \mapsto \hat{f}$ where $\hat{f}$ is any map making the diagram below commute, where $F_2,F'_2 \in \class{F}_2$ and $C_2,C'_2 \in  \widetilde{\class{C}}_2$.
$$\begin{CD}
0 @>>> A       @>>>  F_2 @>>>     C_2 @>>> 0\\
@. @VVfV            @VV \hat{f} V     @.  \\
0 @>>> B       @>>>       F'_2  @>>>     C'_2  @>>> 0 \\
\end{CD}$$
Now applying $\textnormal{Ho}(\class{M}_1/\class{M}_3)  \xrightarrow{\widetilde{Q}_3} \textnormal{Ho}(\class{M}_1)$ to $\hat{f}$ gives us $\bar{f}$ in the next diagram, where $F_3,F'_3 \in \class{F}_3$ and $C_3,C'_3 \in  \widetilde{\class{C}}_3 = \class{C} \cap \class{W}_3$.
$$\begin{CD}
0 @>>> F_3       @>>>  C_3 @>q_A>>     F_2 @>>> 0\\
@. @.          @V\bar{f}VV     @V\hat{f}VV  \\
0 @>>> F'_3       @>>>       C'_3  @>q_B>>     F'_2  @>>> 0 \\
\end{CD}$$ But since $F_2,F'_2 \in \class{F}_2 = \class{W}_3 \cap \class{F}_1$ and $\class{W}_3$ is thick, we see that $F_3,F'_3 \in \class{W}_3 \cap \class{F}_3 = \class{W}_1 \cap \class{F}_1.$ That is, $F_3,F'_3$ are trivial in $\class{M}_1$ and hence the maps $\{q_A\}$ are providing a natural isomorphism $\{q_A\} : \widetilde{Q}_3 \circ R_2 \cong R_2$.

Finally, recall that we have shown the vertical functors in the diagram $(***)$ to be equivalences. So each of these functors are both right and left adjoints of each other. So the isomorphism $\{q_A\} : \widetilde{Q}_3 \circ R_2 \cong R_2$ means that the functor $\textnormal{Ho}(\class{M}_2) \xrightarrow{R_2} \textnormal{Ho}(\class{M}_1)$ has right adjoint the composite $\textnormal{Ho}(\class{M}_1) \xrightarrow{R_1} \textnormal{Ho}(\class{M}_1/\class{M}_3)  \xrightarrow{\widetilde{Q}_3} \textnormal{Ho}(\class{M}_2)$.

\end{proof}

We have the dual notion of left localization, denoted $\class{M}_2\backslash\class{M}_1$, and the dual statements of the above. In particular we have the following left recollement theorem.

\begin{theorem}[Left recollement theorem]\label{them-left recollement}
Let $\cat{A}$ be an abelian category with three hereditary model structures, as below, whose cores all coincide: $$\class{M}_1 = (\class{C}_1, \class{W}_1, \class{F}) , \ \ \ \class{M}_2 = (\class{C}_2, \class{W}_2, \class{F}) , \ \ \ \class{M}_3 = (\class{C}_3, \class{W}_3, \class{F}).$$ If $\class{W}_3 \cap \class{C}_1 = \class{C}_2$ and $\class{C}_3 \subseteq  \class{C}_1$ (or
equivalently, $\widetilde{\class{F}}_2 \cap \class{W}_3 = \widetilde{\class{F}}_1$ and $\class{C}_2 \subseteq \class{W}_3$), then $\class{M}_2\backslash\class{M}_1$ is Quillen equivalent to $\class{M}_3$ and $\class{M}_3\backslash\class{M}_1$ is Quillen equivalent to $\class{M}_2$. In fact, we have a recollement as shown below.
\[
\begin{tikzpicture}[node distance=3.5cm]
\node (A) {$\textnormal{Ho}(\class{M}_2)$};
\node (B) [right of=A] {$\textnormal{Ho}(\class{M}_1)$};
\node (C) [right of=B] {$\textnormal{Ho}(\class{M}_3)$};
\draw[<-,bend left=40] (A.20) to node[above]{$\widetilde{R}_3 \circ Q_1$} (B.160);
\draw[->] (A) to node[above]{\small $Q_2$} (B);
\draw[<-,bend right=40] (A.340) to node [below]{$R_1$} (B.200);

\draw[<-,bend left] (B.20) to node[above]{\small $\lambda = Q_3$} (C.160);
\draw[->] (B) to node[above]{$\bar{\gamma}$} (C);
\draw[<-,bend right] (B.340) to node [below]{\small $\rho = \widetilde{R}_2$} (C.200);
\end{tikzpicture}\]
Here, the functor $\bar{\gamma}$ is the quotient functor of (the dual of) Lemma~\ref{lemma-quotient map}.
\end{theorem}

\section{models for mock homotopy categories of projectives}\label{sec-models for mock}

As described in the introduction we will now give an interesting application of Theorem~\ref{them-left recollement}. In Theorem~\ref{them-recollement of Murfet} we present three interrelated Hovey triples which immediately yield the recollement of Murfet from~\cite{murfet}. We see that the recollement not only holds for many non-noetherian schemes $X$, but we give a vast generalization of the idea in Theorem~\ref{them-generalized recollement} 

Let $(\class{F},\class{C})$ denote the flat cotorsion pair in the category $\qcox$ of quasi-coherent sheaves on $X$, where $X$ is any scheme in which $\qcox$ has a flat generator. This is true for example when $X$ is a quasi-compact semi-separated scheme. Then we already have from~\cite[Theorem~5.5 and Section~5.3]{gillespie-degreewise-model-strucs}, the following two complete hereditary (small) cotorsion pairs in the category $\chqcox$ of chain complexes of quasi-coherent sheaves:
$$(\dwclass{F}, \rightperp{\dwclass{F}}) \ \ \ \ \text{and} \ \ \ \ (\exclass{F}, \rightperp{\exclass{F}}).$$
Here, $\dwclass{F}$ is the class of all complexes $X$ with each $X_n \in \class{F}$, while $\exclass{F}$ is the class of all exact complexes $X$ with each $X_n \in \class{F}$.
Using the characterizations of $\rightperp{\dwclass{F}}$ and $\rightperp{\exclass{F}}$ from~\cite[Propositions~3.2 and~3.3]{gillespie-degreewise-model-strucs} we can easily argue that each of the above cotorsion pairs has core equal to the class of all contractible complexes with components in $\class{F} \cap \class{C}$.
We also have the cotorsion pairs $(\dgclass{F}, \tilclass{C})$ and $(\tilclass{F}, \dgclass{C})$ from~\cite{gillespie} and~\cite{gillespie-quasi-coherent} which have the same cores. The following are immediate corollaries of~\cite[Theorem~1.1]{gillespie-hovey triples}.

\begin{corollary}\label{cor-degreewise flat hovey triple}
There is an hereditary abelian model structure $$\class{M}_1 = (\dwclass{F},\class{W}_1,\dgclass{C})$$ whose core equals the class of all contractible complexes with components in $\class{F} \cap \class{C}$.
\end{corollary}

\begin{proof}
Apply~\cite[Theorem~1.1]{gillespie-hovey triples} to the cotorsion pairs $(\dwclass{F}, \rightperp{\dwclass{F}})$ and $(\tilclass{F}, \dgclass{C})$ to obtain the thick class $\class{W}_1$ yielding the Hovey triple.
\end{proof}

\begin{remark}\label{remark1}
The homotopy category Ho($\class{M}_1$) is equivalent to Murfet's \emph{mock homotopy category of projectives}, $K_m(Proj X)$. To illustrate this, we first note that for any ring $R$, we have from~\cite[Corollary~6.4]{bravo-gillespie-hovey}, the hereditary abelian model structure $(\dwclass{P},\class{W}_1, \class{A})$ on $\ch$. Here $\class{W}_1 = \rightperp{\dwclass{P}}$, where $\dwclass{P}$ is the class of all complexes of projective modules, and $\class{A}$ is the class of all chain complexes. (Following~\cite{positselski}, the complexes in $\class{W}_1$ are called \emph{contraacyclic}.) The corresponding homotopy category is equivalent to $K(Proj) = \dwclass{P}/\sim$, the homotopy category of all complexes of projectives.  Turning to Corollary~\ref{cor-degreewise flat hovey triple}, we note that if $X$ is an affine scheme, the class $\class{W}_1$ of trivial objects is exactly the same. That is, $\class{W}_1 = \rightperp{\dwclass{P}}$. (Reason)  $\rightperp{\dwclass{P}}$ is thick and Neeman has shown in~\cite{neeman-flat} that $\tilclass{F} = \dwclass{F} \cap \rightperp{\dwclass{P}}$. Lemma~\ref{lemma-half related} below now implies that $(\dwclass{F},\rightperp{\dwclass{P}},\dgclass{C})$ is a Hovey triple. But the uniqueness of the class of trivial objects in any Hovey triple implies $\class{W}_1 = \rightperp{\dwclass{P}}$. 
\end{remark}

\begin{lemma}\label{lemma-half related}
Suppose we have two complete cotorsion pairs $(\class{Q},\tilclass{R})$ and  $(\tilclass{Q},\class{R})$ in an abelian category and that we also have a thick class $\class{W}$. Then the following hold. 
\begin{enumerate}
\item  If $\tilclass{Q} = \class{Q} \cap \class{W}$ and $\tilclass{R} \subseteq \class{W}$, then also $\tilclass{R} = \class{W} \cap \class{R}$. That is, $(\class{Q},\class{W},\class{R})$ is a Hovey triple. 
\item  If $\tilclass{R} = \class{W} \cap \class{R}$ and $\tilclass{Q} \subseteq \class{W}$, then also $\tilclass{Q} = \class{Q} \cap \class{W}$. That is, $(\class{Q},\class{W},\class{R})$ is a Hovey triple.
\end{enumerate}
\end{lemma}

\begin{proof}
The statements are proved similarly, and we will prove (1). We have by assumption that $\tilclass{Q} \subseteq \class{Q}$ and consequently $\tilclass{R} \subseteq \class{R}$. We are also assuming $\tilclass{R} \subseteq \class{W}$, and so we have $\tilclass{R} \subseteq \class{W} \cap \class{R}$. 

It is left to show $\tilclass{R} \supseteq \class{W} \cap \class{R}$. Letting $X \in \class{W} \cap \class{R}$, we use completeness of the cotorsion pair $(\class{Q},\tilclass{R})$ to find a short exact sequence $0 \xrightarrow{} X \xrightarrow{} R \xrightarrow{} Q \xrightarrow{} 0$ with $R \in \tilclass{R}$ and $Q \in \class{Q}$. We see that this forces $Q \in \class{Q} \cap \class{W} = \tilclass{Q}$. Hence the sequence must split, forcing $X$ to be a retract of an object in $\tilclass{R}$. So $X$ is also in $\tilclass{R}$. 
\end{proof}

\begin{corollary}\label{cor-exact degreewise flat hovey triple}
There is an hereditary abelian model structure $$\class{M}_2 = (\exclass{F},\class{W}_2,\dgclass{C})$$ whose core equals the class of all contractible complexes with components in $\class{F} \cap \class{C}$.
\end{corollary}

\begin{proof}
Apply~\cite[Theorem~1.1]{gillespie-hovey triples} to the cotorsion pairs $(\exclass{F}, \rightperp{\exclass{F}})$ and $(\tilclass{F}, \dgclass{C})$ to obtain the thick class $\class{W}_2$ yielding the Hovey triple.
\end{proof}

\begin{remark}\label{remark-new}
The homotopy category Ho($\class{M}_2$) is equivalent to Murfet's \emph{mock projective stable derived category}, $K_{m,ac}(Proj X)$. Again, a model was constructed in~\cite[Corollary~6.5]{bravo-gillespie-hovey} for the \emph{projective stable derived category} of an arbitrary ring $R$. It is the hereditary abelian model structure $(\exclass{P},\class{W}_2, \class{A})$ on $\ch$ where $\exclass{P}$ is the class of all exact (acyclic) complexes of projective modules, $\class{W}_2 = \rightperp{\exclass{P}}$, and $\class{A}$ is the class of all chain complexes. The corresponding homotopy category is equivalent to $K_{ac}(Proj) = \exclass{P}/\sim$, the homotopy category of all exact complexes of projectives.
Interpreting Corollary~\ref{cor-exact degreewise flat hovey triple} in the affine case we do in fact have the desired equality $\class{W}_2 = \rightperp{\exclass{P}}$. (Reason)  Similar to the Remark following Corollary~\ref{cor-degreewise flat hovey triple}, it will follow from Lemma~\ref{lemma-half related} by  showing 
$\tilclass{F} = \exclass{F} \cap \rightperp{\exclass{P}}$. But letting $\class{E}$ denote the class of all exact complexes we have $$\exclass{F} \cap \rightperp{\exclass{P}} = (\dwclass{F} \cap \class{E}) \cap \rightperp{\exclass{P}} = \dwclass{F} \cap (\class{E} \cap \rightperp{\exclass{P}}) = \dwclass{F} \cap \rightperp{\dwclass{P}} = \tilclass{F}.$$ The last equality is again using Neeman's result, but the second to last equality also requires justification. To see $\class{E} \cap \rightperp{\exclass{P}} = \rightperp{\dwclass{P}}$ we use yet another application of Lemma~\ref{lemma-half related}. Indeed, 
as shown in~\cite{bravo-gillespie-hovey}, we know that $(\dwclass{P}, \rightperp{\dwclass{P}})$ and $(\exclass{P}, \rightperp{\exclass{P}})$ are each complete cotorsion pairs and of course $\class{E}$ is thick.  
\end{remark}

\begin{corollary}\label{cor-dgflat hovey triple}
Let $\class{E}$ denote the class of all exact complexes.
There is an hereditary abelian model structure $$\class{M}_3 = (\dgclass{F},\class{E},\dgclass{C})$$ whose core equals the class of all contractible complexes with components in $\class{F} \cap \class{C}$.
\end{corollary}

\begin{proof}
This is the flat model structure from~\cite{gillespie} and~\cite{gillespie-quasi-coherent} corresponding to $(\dgclass{F}, \tilclass{C})$ and $(\tilclass{F}, \dgclass{C})$.
\end{proof}

\begin{remark}\label{remark2}
$\class{M}_3$ is a model for $\class{D}(X)$, the derived category of the scheme $X$. This follows from the fact that the trivial objects are precisely the exact complexes. In the affine case we do indeed have $\class{E} = \rightperp{\dgclass{P}}$, where $\dgclass{P}$ is the class of all DG-projective complexes.
\end{remark}

\begin{theorem}\label{them-recollement of Murfet}
Let $X$ be any scheme in which $\qcox$ has a flat generator. Then $\class{M}_1,\class{M}_2,\class{M}_3$ above satisfy the hypotheses of Theorem~\ref{them-left recollement} and recover the recollement of Murfet.
\[
\begin{tikzpicture}[node distance=3.5cm]
\node (A) {$K_{m,ac}(Proj X)$};
\node (B) [right of=A] {$K_m(Proj X)$};
\node (C) [right of=B] {$\class{D}(X)$};
\draw[<-,bend left=40] (A.20) to node[above]{$\widetilde{R}_3 \circ Q_1$} (B.160);
\draw[->] (A) to node[above]{\small $Q_2$} (B);
\draw[<-,bend right=40] (A.340) to node [below]{$R_1$} (B.200);

\draw[<-,bend left] (B.20) to node[above]{\small $\lambda = Q_3$} (C.160);
\draw[->] (B) to node[above]{$\bar{\gamma}$} (C);
\draw[<-,bend right] (B.340) to node [below]{\small $\rho = \widetilde{R}_2$} (C.200);
\end{tikzpicture}\]
\end{theorem}

\subsection{Generalizations}\label{subsec-generalizations}

Rather than $\qcox$, one may instead wish to consider the flat cotorsion pair $(\class{F},\class{C})$ in the category of sheaves of $\class{O}$-modules where $\class{O} = (\class{O}_X,X)$ is a ringed space. In fact, one may wish to start with other Grothendieck categories and cotorsion pairs other than the flat one.

So let $\cat{G}$ be any Grothendieck category and let $(\class{A},\class{B})$ be a complete hereditary cotorsion pair. Then from~\cite[Proposition~3.2 and~3.3]{gillespie-degreewise-model-strucs} we have the first two cotorsion pairs below, and from~\cite[Proposition~3.6]{gillespie} we have the second two cotorsion pairs below.
$$(\dwclass{A}, \rightperp{\dwclass{A}}) \ , \ \ \ \ (\exclass{A}, \rightperp{\exclass{A}}) \ , \ \ \ \ (\dgclass{A},\tilclass{B}) \ , \ \ \ \ (\tilclass{A},\dgclass{B})$$
The definitions are analogous to those above using the flat cotorsion pair $(\class{F},\class{C})$. They are all hereditary since $(\class{A},\class{B})$ is assumed to be. Moreover, just as in the flat situation, we can easily argue that each of the above cotorsion pairs has core equal to the class of all contractible complexes with components in $\class{A} \cap \class{B}$.

Now assume $(\class{A},\class{B})$ is cogenerated by a set $\class{S}$ and that $\class{A}$ contains a generating set $\{U_i\}$ for $\cat{G}$. The methods of~\cite{gillespie-degreewise-model-strucs}, \cite[Proposition~3.6]{gillespie}, and~\cite{gillespie-quasi-coherent} proving completeness of the four associated cotorsion pairs above rely on the extra assumption that $\class{A}$ be closed under direct limits. Others have since removed this assumption, giving us the following lemma.

\begin{lemma}\label{lemma-completeness of cotorsion pairs}
Let $\cat{G}$ be any Grothendieck category and $(\class{A},\class{B})$ an hereditary cotorsion pair cogenerated by some set $\class{S}$ and such that $\class{A}$ contains a generating set $\{U_i\}$ for $\cat{G}$. Then the four induced cotorsion pairs on $\cha{G}$ listed above are also cogenerated by a set and complete cotorsion pairs.
\end{lemma}

\begin{proof}
We note that in the case of a concrete Grothendieck category, Theorem~3.1 of~\cite{enochs-model strucs} immediately gives us the result. For a general Grothendieck category we can argue as follows. First, note that $\{D^n(U_i)\}$ is a generating set for $\cha{G}$ and is contained in each of the four classes $\dwclass{A}$, $\exclass{A}$, $\dgclass{A}$, and $\tilclass{A}$. By Proposition~\ref{prop-cotorsion pairs and deconstructibility} we see that $\class{A}$ is deconstructible and that it is enough to show each of these four classes is also deconstructible. But {\v{S}}\v{t}ov{\'{\i}}{\v{c}}ek has shown in~\cite[Theorem~4.2]{stovicek} that since $\class{A}$ is deconstructible we automatically must have that $\dwclass{A}$, $\dgclass{A}$, and $\tilclass{A}$ are each deconstructible. This proves the lemma for three of the four cotorsion pairs but we still must deal with $\exclass{A}$. Here we note that $\exclass{A} = \dwclass{A} \cap \class{E}$ where $\class{E}$ is the class of all exact complexes. However, in any Grothendieck category $\cat{G}$, the class of all objects is itself deconstructible and from this it follows from~\cite[Theorem~4.2~(2)]{stovicek} that $\class{E}$ is also deconstructible. We now turn around and conclude that the intersection of the two deconstructible classes $\exclass{A} = \dwclass{A} \cap \class{E}$ is also deconstructible by~\cite[Proposition~2.9~(2)]{stovicek}.
\end{proof}

 It now follows from the above that Corollaries~\ref{cor-degreewise flat hovey triple}, \ref{cor-exact degreewise flat hovey triple}, and~\ref{cor-dgflat hovey triple} all have direct generalizations from the flat cotorsion pair $(\class{F},\class{C})$ to more general hereditary cotorsion pairs $(\class{A},\class{B})$. This leads us to the following generalization of Theorem~\ref{them-recollement of Murfet}

\begin{theorem}\label{them-generalized recollement}
Let $\cat{G}$ be any Grothendieck category and $(\class{A},\class{B})$ an hereditary cotorsion pair cogenerated by some set $\class{S}$ and such that $\class{A}$ contains a generating set $\{U_i\}$ for $\cat{G}$. Then we have three hereditary abelian model structure
$$\class{M}_1 = (\dwclass{A},\class{W}_1,\dgclass{B}) \ , \ \ \class{M}_2 = (\exclass{A},\class{W}_2,\dgclass{B}) \ , \ \ \class{M}_3 = (\dgclass{A},\class{E},\dgclass{B}).$$
They satisfy that each has core equaling the class of all contractible complexes with components in $\class{A} \cap \class{B}$ and so Theorem~\ref{them-left recollement} applies to yield a recollement.
\end{theorem}

\end{document}